\documentclass[a4paper, 11pt]{article}
\usepackage[mathscr]{eucal}
\usepackage{amssymb}
\usepackage{latexsym}
\usepackage{amsthm}
\usepackage{amsmath}
\usepackage[dvips]{graphicx}
\usepackage{psfrag}
\usepackage{a4wide}

\newtheorem{theorem}{Theorem}

\newtheorem{corollary}{Corollary}
\theoremstyle{definition}

\newtheorem{remark}{Remark}

\title{Linear programming bounds for regular graphs} 

\author{
Hiroshi Nozaki
}
\begin{document}
\maketitle

\renewcommand{\thefootnote}{\fnsymbol{footnote}}
\footnote[0]{
%05D05  	Extremal set theory
\\
\noindent
{\it Hiroshi Nozaki}: 
	Department of Mathematics Education, 
	Aichi University of Education
	1 Hirosawa, Igaya-cho, 
	Kariya, Aichi 448-8542, 
	Japan.
	hnozaki@auecc.aichi-edu.ac.jp
}

\begin{abstract}
Delsarte, Goethals, and Seidel (1977) used  the linear 
programming method in order to find bounds for the size of spherical codes endowed with prescribed inner products between distinct points in the code.  
In this paper, we develop the linear programming method to obtain bounds for the number of vertices of  connected regular graphs endowed with given distinct eigenvalues. 
This method is proved by some ``dual'' technique of the spherical case, motivated from the theory of association scheme.   
As an application of this bound, we prove that a connected $k$-regular graph satisfying $g>2d-1$ has the minimum second-largest eigenvalue of all $k$-regular graphs of the same size, where $d$ is the number of distinct non-trivial eigenvalues, and $g$ is the girth. The known graphs satisfying $g>2d-1$ are Moore graphs, incidence graphs of regular generalized polygons of order $(s,s)$,  triangle-free strongly regular graphs,
and the odd graph of degree $4$. 
\end{abstract}

\section{Introduction}
Delsarte \cite{D73_2} has introduced the linear programming method to find bounds for the size of codes with prescribed distances over finite field. 
This is called Delsarte's method, and he stated it for codes in certain special association schemes, so called $Q$-polynomial schemes, including 
the Johnson scheme and the Hamming scheme. 
Delsarte, Goethals, and Seidel \cite{DGS77} gave the linear programming method on the Euclidean sphere. This is naturally generalized to the compact two-point homogeneous spaces \cite{KL78}. 
Delsarte's method is also extended to various situations like the permutation codes \cite{T99}, the Grassmannian codes \cite{B06}, or the ordered codes \cite{BP09}. The linear programming is very powerful to solve optimization problems, 
for instance maximizing the size of codes for given distances \cite{OS79,M08}, or maximizing the minimum distance for a fixed cardinality \cite{L92,CK07}. 
In the present paper, we develop the linear programming method to find bounds for the 
order of  
connected regular graphs with given distinct eigenvalues.   This method is not based on Delsarte's 
but a kind of ``dual'' technique of the spherical case 
inspired from the theory of association schemes. 

Let $X$ be a finite set, and $R_0, \ldots, R_d$ symmetric binary relations on $X$. 
The {\it $i$-th adjacency matrix} $\boldsymbol{A}_i$ is defined to be the matrix indexed by $X$ whose 
 $(x,y)$-entry is 1 if $(x,y) \in R_i$, 0 otherwise. 
A configuration $\mathfrak{X}=(X,\{R_i\}^{d}_{i=0})$ is called a {\it symmetric association scheme} of class $d$ if $\{\boldsymbol{A}_i \}_{i=0}^d$ satisfies the following: 
(1) $\boldsymbol{A}_0=\boldsymbol{I}$ (identity matrix),
(2) $\sum_{i=0}^d \boldsymbol{A}_i=\boldsymbol{J}$ (all-ones matrix),
(3) there exist  real numbers $p_{ij}^k$ such that $\boldsymbol{A}_i\boldsymbol{A}_j =\sum_{k=0}^d p_{ij}^k\boldsymbol{A}_k$ for all $i,j \in \{0,1,\ldots,d\}$. 
The vector space $\mathfrak{A}$ spanned by $\{\boldsymbol{A}_i \}_{i=0}^d$
over $\mathbb{C}$ forms a commutative algebra, and it is called 
the {\it Bose--Mesner algebra} of $\mathfrak{X}$. It is well known
that $\mathfrak{A}$ is semi-simple \cite[Section 2.3, II]{BIb},
hence it has  the primitive 
idempotents $\boldsymbol{E}_0=(1/|X|)\boldsymbol{J}, \boldsymbol{E}_1, \ldots, \boldsymbol{E}_d$, which form a basis of $\mathfrak{A}$. 

We have two remarkable classes of association schemes so called $P$-polynomial association schemes and 
$Q$-polynomial association schemes. An association scheme is said to be 
{\it $P$-polynomial} if for each $i\in \{0,1, \ldots, d\}$ there exists 
a polynomial $v_i$ of degree $i$ such that $\boldsymbol{A}_i=v_i(\boldsymbol{A}_1)$. 
A $P$-polynomial scheme has the relations as the path distances of the graph $(X,R_1)$, and $(X,R_1)$ becomes a distance-regular graph \cite{BCNb}. 
An association scheme is said to be 
{\it $Q$-polynomial} if for each $i\in \{0,1, \ldots, d\}$ there exists 
a polynomial $v^{*}_i$ of degree $i$ such that $|X|\boldsymbol{E}_i=v_i^*(|X|\boldsymbol{E}_1^\circ)$, where $\circ$ means  
the multiplication is the entry-wise product. 
Roughly speaking, the $P$-polynomial schemes and the $Q$-polynomial schemes correspond to discrete cases of 
the concepts of two-point homogeneous spaces and rank 1 symmetric spaces, respectively \cite[Section 3.6, III]{BIb}, \cite[Chapter 9]{CSb}. 

By swapping the matrix multiplication~$\cdot$ and the entry-wise multiplication~$\circ$, 
the bases $\{\boldsymbol{A}_i\}_{i=0}^d$ and $\{\boldsymbol{E}_i\}_{i=0}^d$ very similarly behave in 
the Bose--Mesner algebra. The following are basic equations for 
the bases \cite[Section 2.2, 2.3, II]{BIb}: 
\begin{align}
&\sum_{i=0}^d \boldsymbol{A}_i =\boldsymbol{J}=|X|\boldsymbol{E}_0, & &\sum_{i=0}^d \boldsymbol{E}_i=\boldsymbol{I}=\boldsymbol{A}_0, \label{1.1}\\ 
&\boldsymbol{A}_i \circ \boldsymbol{A}_j = \delta_{ij}\boldsymbol{A}_i, &
& \boldsymbol{E}_i \cdot \boldsymbol{E}_j = \delta_{ij}\boldsymbol{E}_i, \label{1.2} \\
&\boldsymbol{A}_i \cdot \boldsymbol{A}_j = \sum_{k=0}^d p_{ij}^k \boldsymbol{A}_k, &
&\boldsymbol{E}_i \circ \boldsymbol{E}_j = \frac{1}{|X|}\sum_{k=0}^d q_{ij}^k \boldsymbol{E}_k, \label{1.3} \\
&\boldsymbol{A}_i= \sum_{j=0}^d P_i(j)\boldsymbol{E}_j,&  &\boldsymbol{E}_i=\frac{1}{|X|}\sum_{j=0}^d Q_i(j) \boldsymbol{A}_j, \label{1.4} 
\end{align}
\begin{align}
&\boldsymbol{A}_i\cdot \boldsymbol{J} =k_i \boldsymbol{J}, & &|X|\boldsymbol{E}_i \circ \boldsymbol{I}= m_i \boldsymbol{I}, \label{1.5} \\
&\boldsymbol{A}_i \circ \boldsymbol{I}=0\ (i\ne 0), & & \boldsymbol{E}_i \cdot \boldsymbol{J}=0\ (i\ne 0), \label{1.9} \\
&\tau(\boldsymbol{A}_i)=|X|k_i, & &{\rm tr}(\boldsymbol{E}_i)=m_i, \label{1.6}\\
&{\rm tr} (\boldsymbol{A}_i)=0\ (i\ne 0), & &\tau(\boldsymbol{E}_i)=0\ (i \ne 0), \label{1.8}\\
&k_0=1, & &m_0=1, \label{1.7}
\end{align}
where $\delta_{ij}$ denotes the Kronecker delta,
$\tau(\boldsymbol{M})$ denotes the summation of all entries in $\boldsymbol{M}$,
$k_i$ is the degree of the graph $(X,R_i)$, 
and $m_i$ is the rank of $\boldsymbol{E}_i$. 
Here $p_{ij}^k$ is called the {\it intersection number}, and it is equal to
the size of $\{z \in X \mid (x,z) \in R_i, (z,y) \in R_j\}$ with 
$(x,y) \in R_k$. Naturally $p_{ij}^k$ is a non-negative integer. On the other hand, 
$q_{ij}^k$ is called the {\it Krein number}, it can be proved that 
it is a non-negative real number \cite{S73}.  Such a kind of similar properties 
obtained by swapping $\boldsymbol{A}_i$, $\boldsymbol{E}_i$, multiplications, and corresponding parameters 
is called a {\it dual property}. It is obviously seen that 
the dual concept of $P$-polynomial scheme is $Q$-polynomial scheme. 
There are a number of non-trivial dual properties between $P$-polynomial schemes 
and $Q$-polynomial schemes \cite{MT09,BIb}, and several conjectures are still left \cite{MT09}. 

The matrix $\boldsymbol{A}_i$ is regarded as a regular 
graph. The matrix $\boldsymbol{E}_i$ is positive semidefinite with equal diagonals, and it is 
interpreted as a spherical set. 
 We can observe the dual relationship of 
$\boldsymbol{A}_i$ and $\boldsymbol{E}_i$ in the Bose--Mesner algebra, and 
it  
shows how properties of graphs dually correspond to those of spherical sets. 
For example, \eqref{1.4} says that eigenvalues of graphs dually correspond to inner products in spherical sets. 
Table~\ref{tb:1} shows the dual correspondence of  the properties of graphs and spherical sets. 

\begin{table}\caption{Dual properties between regular graph and spherical set}\label{tb:1}
\begin{center}
\begin{tabular}{c|c||c} 
 regular graph  &  spherical set & reason\\
\hline 
$\boldsymbol{A}$: adjacency matrix & $\boldsymbol{E}$: Gram matrix &bases\\
$k$: degree & $m$: dimension &\eqref{1.5},\eqref{1.6}\\
regular & spherical& \eqref{1.5} \\
eigenvalues & inner products& \eqref{1.4} \\ 
connected & constant weight& \eqref{1.7} \\
no loop & spherical 1-design & \eqref{1.9} \\
no multiple edge & spherical 2-design & \eqref{1.2}\\
$\tau$ & ${\rm tr}$& \eqref{1.6}, \eqref{1.8} \\ 
Moore graph & tight spherical design & tightness \\
$F_i^{(k)}(x)$ & $\mathcal{Q}_i^{(m)}(x)$ & \\
${\rm tr}(F_i^{(k)}(\boldsymbol{A})) \geq 0$& $\tau(\mathcal{Q}_i^{(m)}(\boldsymbol{E}^\circ))\geq 0$ & \cite{S66,DGS77} \\
 ${\rm tr}(F_i^{(k)}(\boldsymbol{A})) = 0$ for $1\leq i \leq g-1$ &$\tau(\mathcal{Q}_i^{(m)}(\boldsymbol{E}^\circ)) = 0$ for $1\leq i \leq t$ &\\
$\Leftrightarrow$ girth $g$ & $\Leftrightarrow$ spherical $t$-design &\cite{S66,DGS77}
\end{tabular}
\end{center}
\end{table}

We can interpret the Euclidean sphere $S^{m-1}$ as a continuous case of $Q$-polynomial scheme. 
The polynomial $v_i^*$ on a $Q$-polynomial scheme corresponds to 
the Gegenbauer polynomial $\mathcal{Q}^{(m)}_i$ on $S^{m-1}$ \cite{DGS77}. 
We have fundamental parameters $s$ and $t$ for a finite subset $X$ in $S^{m-1}$. 
The parameter $s$ is just the number of the Euclidean distances between distinct points in $X$. If $X$ has only $s$ distances, then we have 
\begin{equation} \label{eq:s-dis}
|X| \leq \binom{m+s-1}{s}+\binom{m+s-2}{s-1}. 
\end{equation}
The other parameter $t$ is the strength in the sense of spherical design. We call $X$ 
a spherical $t$-design if for any polynomial $f$ in $m$ variables of degree at most $t$ the following equation 
holds: 
\[
\frac{1}{|S^{m-1}|}\int_{S^{m-1}} f(x) dx= \frac{1}{|X|} \sum_{x \in X} f(x),
\]
where $|S^{m-1}|$ is the volume of $S^{m-1}$. 
One of unexpected results is that if $t\geq 2s-2$ holds, then $X$ has the structure of 
a $Q$-polynomial scheme with the relations of distances \cite{DGS77}. 
For a spherical $2e$-design $X$ in $S^{m-1}$, we have an absolute bound \cite{DGS77}:  
\[
|X| \geq \binom{m+e-1}{e} + \binom{m+e-2}{e-1}. 
\] 
A spherical design is said to be {\it tight} if it attains this equality. 
A tight design satisfies $t=2s$ \cite{DGS77}, and hence it becomes a $Q$-polynomial scheme. A tight design also attains the bound \eqref{eq:s-dis}. 
Moreover the polynomial $v_i^{\ast}$
of a tight design coincides with $\mathcal{Q}_i^{(m)}$.

For connected regular graph, we have a very similar situation to the above argument on the sphere. Let $G$ be a connected $k$-regular graph with $v$ vertices. Throughout this paper, 
we assume a graph is simple. 
Since a graph with $d+1$ distinct eigenvalues is of diameter at most $d$, 
we can change the assumption of the Moore bound to the number of eigenvalues. 
Namely if $G$ has only   
$d+1$ distinct eigenvalues, then we have  
\[
v \leq 1+k\sum_{j=0}^{d-1}(k-1)^j. 
\]
  If equality holds, then 
$G$ is called a {\it Moore graph}. Tutte \cite{T66} showed that if $G$ is of girth $2e+1$, then we have 
\[
v \geq 1+k\sum_{j=0}^{e-1}(k-1)^j.
\]
The graph that attains this equality becomes a Moore graph. It is well known that a Moore graph is distance-regular.
Actually we can show that if $g\geq 2d-1$ holds, then $G$ is distance-regular (Theorem~\ref{thm:girth}).   
Let $F_i^{(k)}$ be the polynomial of degree $i$ defined by \eqref{eq:F_012} and \eqref{eq:F_i} in 
Section~\ref{sec:2}.  
The polynomial $v_i$ of a $k$-regular Moore 
graph coincides with $F_i^{(k)}$. 
Apparently the dual concept of tight spherical design is Moore graph, and the polynomial $F_i^{(k)}$ dually corresponds to the Gegenbauer polynomial $\mathcal{Q}_i^{(m)}$. 

The linear programming method for spherical codes is 
essentially based on the positive definiteness of the Gegenbauer polynomials, namely 
$\tau(\mathcal{Q}_i^{(m)}(\boldsymbol{E}^\circ)) \geq 0$, where 
$\boldsymbol{E}$ is the Gram matrix. In this paper, 
 we dually show the linear programming method for connected regular graphs by using the property ${\rm tr}(F_i^{(k)}(\boldsymbol{A})) \geq 0$, where $\boldsymbol{A}$ is the adjacency matrix.     

We can apply the linear programming method for determining the graph maximizing the spectral gap.
The {\it spectral gap} of a graph is the difference 
between the first and second largest 
eigenvalues of the graph. 
The {\it edge expansion ratio} $h(G)$ of a $k$-regular graph $G=(V,E)$ is defined as 
\[
h(G)=\min_{S\subset V, |S|\leq |V|/2} \frac{|\partial S|}{|S|},
\]
where $\partial S=\{\{u,v\} \mid u \in S, v\in V\setminus S, \{u,v\} \in E  \}$. 
By the spectral gap $\tau$ of $G$, we have 
$\tau/2 \leq h(G) \leq \sqrt{2k\tau}$ \cite{A86,AM85,D84}. 
This  implies that a graph with large spectral gap has high connectivity.  
The second-largest eigenvalue cannot be much smaller than $2\sqrt{k-1}$ \cite{A86}. 
Ramanujan graphs have an asymptotically smallest possible second-largest 
eigenvalue (see \cite{HLW06}).  
Several regular graphs with very small second-largest eigenvalues 
are determined (see \cite{KS13}). 

The dual concept of the graphs maximizing the spectral gap is well known as optimal 
spherical code in the sense of maximizing the minimum distance. 
The optimal configurations of $n$ points on $S^2$ are known only for $n\leq 13$, and $n=24$ \cite[Chapter 3]{EZ01}, \cite{MT12}. 
For higher dimensions, the linear or semidefinite programming bound determined many optimal codes \cite{L92,CK07,BV09}. In particular, we have a strong theorem using the parameters $s$ and $t$, namely if $t\geq 2s-1$ holds, then the set is optimal \cite{L92,CK07}. In the present paper, 
as the dual theorem of it, we prove that a connected $k$-regular graph satisfying $g\geq 2d$ has the minimum second-largest eigenvalue of all $k$-regular graphs of the same size, where $d$ is the number of distinct non-trivial eigenvalues, and $g$ is the girth.

\section{Linear programming method} \label{sec:2}
In the present section, we give the linear programming bounds for 
connected regular graphs. First let us introduce certain polynomials 
$F_i^{(k)}(x)$ which play a key role in the linear programming method. Indeed $F_i^{(k)}(x)$ is the polynomial attached to the homogeneous tree of degree $k$, which is an infinite distance-regular graph.

A graph $G=(V,E)$ is said to be {\it locally finite} if the
degree of any vertex is finite. We also consider an infinite graph here. 
A {\it path} in a graph is a sequence of vertices, where any two consecutive vertices are connected. 
Let $d(x,y)$ be the shortest path distance
from $x\in V$ to $y \in V$. 
The {\it $i$-th distance matrix} $\boldsymbol{A}_i$ of $G$ is defined to be the matrix indexed by 
$V$ whose $(x,y)$-entry is 1 if $d(x,y)=i$, 0 otherwise. 
In particular $\boldsymbol{A}_1$ is called the {\it adjacency matrix} of $G$. 

A locally finite graph $G=(V,E)$ is called a {\it distance-regular graph} 
if for any choice of $x,y \in V$ with $d(x,y)=k$, the number of vertices
$z\in V$ such that $d(x,z)=i$, $d(z,y)=j$ is independent 
of the choice $x, y$.  
For $x,y \in V$ with $d(x,y)=k$, the number 
$
p_{ij}^k=|\{z \in V | d(x,z)=i, d(z,y)=j\}|
$
is called the {\it intersection number} of a distance-regular graph. We use the notation $a_i=p_{1,i}^i$, $b_i=p_{1,i+1}^i$, and $c_i=p_{1,i-1}^i$. 
The {\it intersection array} of a distance-regular graph is defined to be 
\[
\begin{pmatrix}
\ast & c_1& c_2& \cdots\\
a_0&a_1 &a_2 & \cdots \\
b_0&b_1 &b_2 & \cdots 
\end{pmatrix}.\]
The matrix $\boldsymbol{A}_i$ of a distance-regular graph can be written as the polynomial $v_i$ in $\boldsymbol{A}_1$ of degree $i$ \cite[page 127]{BCNb}, where $v_i$ is defined by
\begin{align*}
&v_0(x)=1, \qquad v_1(x)=x,\\
&c_{i+1}v_{i+1}(x)=(x-a_i)v_i(x)- b_{i-1}v_{i-1}(x) \qquad (i=1,2,\ldots). 
\end{align*}

A homogeneous tree of degree $k$ is an infinite distance-regular graph with  intersection numbers
\[
b_0=k, \quad
b_i=k-1 (i=1,2,\ldots), \quad
c_i=1 (i=1,2,\ldots), \quad
a_i=0 (i=0,1,2,\ldots). 
\]
Let $F_i^{(k)}$ denote a polynomial of degree $i$ defined by: 
\begin{equation}\label{eq:F_012}
F_0^{(k)}(x)=1, \qquad F_1^{(k)}(x)=x, \qquad  F_2^{(k)}(x)=x^2 - k,   
\end{equation}
and
\begin{equation}\label{eq:F_i}
F_i^{(k)}(x)=x F_{i-1}^{(k)}(x)- (k-1) F_{i-2}^{(k)}(x)
\end{equation}
for $i\geq 3$. 
Let $q=\sqrt{k-1}$. The polynomials $F_i^{(k)}$ form a sequence of orthogonal polynomials with respect to the weight 
\begin{equation} \label{eq:w(x)}
w(x)=\frac{\sqrt{4q^2-x^2}}{k^2-x^2} 
\end{equation}
on the interval $[-2q, 2q]$ (see \cite[Section~4]{HOb}). 
Note that $F_i^{(k)}(k)=k(k-1)^{i-1}$ for any $i\geq 1$.

A path $u_0 \sim u_1 \sim \cdots \sim u_p$ is said to be {\it reducible} if any sequence 
$u_i \sim u_{j} \sim u_i$ appears \cite{S66}. A path is said to be {\it irreducible} if 
the path is not reducible. 
\begin{theorem}[{\cite{S66}}] \label{thm:red_path}
Let $G$ be a connected $k$-regular graph with adjacency matrix $\boldsymbol{A}$. Then 
the $(u,v)$-entry of $F_i^{(k)}(\boldsymbol{A})$ is the number of irreducible paths 
of length $i$ from $u$ to $v$. 
\end{theorem}
By Theorem~\ref{thm:red_path}, the following is obvious. 
\begin{corollary} \label{coro:girth}
Let $G$ be a connected $k$-regular graph with adjacency matrix $\boldsymbol{A}$. 
Then the following are equivalent. 
\begin{enumerate}
\item ${\rm tr}( F_i^{(k)}(\boldsymbol{A}))=0$ for each $1\leq i \leq g-1$, and ${\rm tr}( F_{g}^{(k)}(\boldsymbol{A}))\ne 0$. 
\item $G$ is of girth $g$.
\end{enumerate}
\end{corollary}

The following is the linear programming bound for connected regular graphs. 
\begin{theorem}\label{thm:lp_bound}
Let $G$ be a connected $k$-regular graph with $v$ vertices. 
Let $\tau_0=k, \tau_1, \ldots, \tau_d $ be the distinct eigenvalues of $G$. 
Suppose there exists a polynomial $f(x)=\sum_{i\geq 0}  f_i F_i^{(k)}(x)$ such that  
$f(k)>0$, $f(\tau_i) \leq 0$ for any $i\geq 1$, $f_0>0$, and $f_i \geq 0$ for any $i\geq 1$. Then we have
\begin{equation} \label{eq:lp} 
v \leq \frac{f(k)}{f_0}. 
\end{equation}
\end{theorem}
\begin{proof}
Let $\boldsymbol{A}$ be the adjacency matrix of $G$. From the spectral decomposition 
$\boldsymbol{A}=\sum_{i=0}^d \tau_i \boldsymbol{E}_i$, we have 
\begin{equation} \label{eq:lp_1}
\sum_{i=0}^d f(\tau_i)\boldsymbol{E}_i =f(\boldsymbol{A})=\sum_{i\geq 0} f_i F_i^{(k)}(\boldsymbol{A})
=f_0 \boldsymbol{I} + \sum_{i\geq 1} f_i F_i^{(k)}(\boldsymbol{A}). 
\end{equation}
Taking the traces in  \eqref{eq:lp_1}, we have  
\[
f(k)={\rm tr} (f(k)\boldsymbol{E}_0)
 \geq {\rm tr}( \sum_{i=0}^d f(\tau_i)\boldsymbol{E}_i) 
={\rm tr}(f_0 \boldsymbol{I} + \sum_{i\geq 1} f_i F_i^{(k)}(\boldsymbol{A}))
\geq {\rm tr}(f_0 \boldsymbol{I})=vf_0.
\] 
Therefore we have $
v \leq f(k)/f_0
$.
\end{proof}

\begin{remark}
We can normalize $f_0=1$ in Theorem~\ref{thm:lp_bound}.
\end{remark}

\begin{remark} \label{rem:attain}
Let $f$ be a polynomial which satisfies the condition in 
Theorem~\ref{thm:lp_bound}. 
 The equality holds in \eqref{eq:lp} if and only if 
$f_i{\rm tr} (F_i^{(k)}(\boldsymbol{A}))=0$ for any $i=1,\ldots, {\rm deg}(f)$, and $f(\tau_i)=0$ for any $i=1,\ldots, d$. In particular, if $f_i>0$ for any $i$, then the girth of $G$ is at least 
${\rm deg}(f)+1$ by Corollary~\ref{coro:girth}. 
\end{remark}

\begin{remark}
Theorem \ref{thm:lp_bound} can be expressed as 
the following linear programming problem and its dual.
{\small
\[
v \leq \max_{m_i} \left\{1+m_1+\cdots+m_d \mid 
\begin{array}{cc}
-\sum_{i=1}^d m_i F_j^{(k)}(\tau_i)\leq F_j^{(k)}(k), & j=1,\ldots,u, \\
m_i\geq 0, & i=1,\ldots ,d 
\end{array}
\right\},
\]
\[
v \leq \min_{f_j} \left\{ 1+f_1F_1^{(k)}(k)+\cdots+f_uF_u^{(k)}(k) \mid 
\begin{array}{cc}
-\sum_{j=1}^u f_j F_j^{(k)}(\tau_i) \geq 1, & 
i=1,\ldots,d,\\
 f_j\geq 0, &j=1,\ldots, u
\end{array}
\right\},
\]
}
where $u$ is the degree of $f$, $m_i$ is the multiplicity of $\tau_i$ and $f_0=1$.
\end{remark}
\begin{remark}
Delsarte, Goethals, and Seidel \cite{DGS77} gave the 
linear programming bounds for spherical codes by using inner products and Gegenbauer polynomials,  instead of
eigenvalues and $F_i^{(k)}$. This is the dual version of 
Theorem~\ref{thm:lp_bound}.
\end{remark}

\section{Minimizing the second-largest eigenvalue}
%Let $\mathfrak{G}(v,k)$ be a set of all  $k$-regular graphs with $v$ vertices. For fixed $v$ and $k$, we would like to find or classify the graphs having the minimum second-largest eigenvalue in $\mathfrak{G}(v,k)$. A graph having the minimum second-largest eigenvalue is called 
%an {\it extremal expander graph}.  
For fixed $v$ and $k$, a graph $G$ is said to be {\it extremal expander} 
if $G$ has the minimum second-largest eigenvalue in
all $k$-regular graphs of order $v$. 
A disconnected graph is not extremal expander, because the
first and second largest eigenvalues are equal. 
In the present section, we obtain extremal expander graphs for several $v$ and $k$ by applying the linear programming method.  First we give several results related to $F_i^{(k)}(x)$. 

\begin{theorem} \label{thm:posi_coef}
Let $F_i^{(k)}(x)F_j^{(k)}(x)=\sum_{l=0}^{i+j} p_l(i,j)F_l^{(k)}(x)$
for real numbers $p_l(i,j)$. 
Then we have
$
p_0(i,j)=F_i^{(k)}(k)\delta_{ij}$ and   $p_l(i,j) \geq 0
$
for all $l,i,j$. Moreover
$p_l(i,j)>0$ if and only if $|i-j|\leq l \leq i+j$ and $l \equiv i+j \pmod{2}$. 
\end{theorem}
\begin{proof}
Let $T_k$ be a homogeneous tree of degree $k$. 
Let $\boldsymbol{A}_i$ be the $i$-th distance matrix of $T_k$, 
and $p_{ij}^k$ the intersection number of $T_k$. 
Since $F_i^{(k)}(x)$ is the polynomial attached to $T_k$, we have
\begin{equation}\label{eq:AsAt}
\sum_{l=0}^{i+j} p_l(i,j)\boldsymbol{A}_l=\sum_{l=0}^{i+j} p_l(i,j)F_l^{(k)}(\boldsymbol{A}_1)=F_i^{(k)}(\boldsymbol{A}_1)F_j^{(k)}(\boldsymbol{A}_1)=\boldsymbol{A}_i\boldsymbol{A}_j=
\sum_{l=0}^{i+j} p_{ij}^l\boldsymbol{A}_l.
 \end{equation}
Clearly $p_l(i,j)=p_{ij}^l$ holds. 
This theorem now follows by a counting argument. 
\end{proof}

Since $F_{i+1}^{(k)}(k)-(k-1)F_i^{(k)}(k)=0$ holds, let $G_i^{(k)}(x)$ denote the polynomial of degree $i$
\[
G_i^{(k)}(x)=\frac{F_{i+1}^{(k)}(x)-(k-1)F_i^{(k)}(x)}{x-k} 
\]
for any $i\geq 1$, and 
$G_0^{(k)}(x)=1$. 
By the three-term recurrence relation \eqref{eq:F_i}, it holds that
\[
G_i^{(k)}(x)=\sum_{j=0}^iF_j^{(k)}(x). 
\]
From Lemmas~3.3, 3.5 in \cite{CK07}, $G_0,G_1,\ldots$ are monic orthogonal polynomials with respect to the positive weight
$
u(x)=(k-x)w(x)
$ 
on the interval $[-2q, 2q]$, where 
$w(x)$ is defined in \eqref{eq:w(x)}.  

\begin{theorem}[{\cite[Theorem~3.1]{CK07}}] \label{thm:CK}
Let $p_0,p_1, \ldots$ be  monic orthogonal polynomials with ${\rm deg}(p_i)=i$.  
Then for any $\alpha\in \mathbb{R}$, the polynomial 
$p_n+\alpha p_{n-1}$ has $n$ distinct real roots 
$
r_1<\cdots<r_n$. 
Moreover for $k<n$, 
$
\prod_{i=1}^k (x-r_i)
$
has positive coefficients in terms of $p_0(x),p_1(x), \ldots, p_k(x)$. 
\end{theorem}

The following is a key theorem. 
\begin{theorem}\label{thm:expander}
Let $G$ be a connected $k$-regular graph of girth $g$.  
Assume the number of  distinct eigenvalues of $G$ is $d+1$. 
If $g \geq 2d$ holds, then $G$ is an extremal expander graph. 
\end{theorem}
\begin{proof}
Let $\tau_0=k>\tau_1>\ldots>\tau_{d}$ be the distinct eigenvalues of $G$. 
We show the polynomial 
\[
f(x)=(x-\tau_1)\prod_{i=2}^d(x-\tau_i)^2=\sum_{i=0}^{2d-1}f_i F_i^{(k)}(x)
\]
satisfies the condition in Theorem~\ref{thm:lp_bound}. 
It trivially holds that $f(k)>0$, and $f(\tau_i)=0$ for any $i=1,\ldots,d$.

Let $\boldsymbol{A}$ be the adjacency matrix of $G$.  
If the diameter of $G$ is greater than $d$, then 
the number of distinct eigenvalues is greater  than $d+1$. Thus 
 the diameter of $G$ is at most $d$. 
Since  $g \geq 2d$ holds, the diameter is exactly $d$. 
Then $G$ partially has  the structure of a homogeneous tree around any vertex, namely $F_i^{(k)}(\boldsymbol{A})=\boldsymbol{A}_i$ for any $i=0,1,\ldots,d-1$. 
Because the Hoffman polynomial \cite{H63} of $G$ is of degree $d$, there exists a natural number $e$ such that 
\[
\sum_{i=0}^{d-1}F_i^{(k)}(\boldsymbol{A})+\frac{1}{e} F_d^{(k)}(\boldsymbol{A})=\boldsymbol{J},
\]
where $\boldsymbol{J}$ is the all-ones matrix. 
 Note that the roots of the Hoffman polynomial 
$P(x)=\sum_{i=0}^{d-1}F_i^{(k)}(x)+(1/e) F_d^{(k)}(x)$ 
  are the non-trivial distinct eigenvalues of $G$ \cite{H63}.

For some positive constant number $c$, the polynomial $f(x)$ can be expressed as 
\[
f(x)= \frac{c P(x)^2}{x-\tau_1}
= \frac{c}{e} \frac{G_d^{(k)}(x)-(1-e) G_{d-1}^{(k)}(x)}{x-\tau_1}P(x).
\]
By Theorem~\ref{thm:CK}, 
$g(x)=(G_d^{(k)}(x)-(1-e) G_{d-1}^{(k)}(x))/(x-\tau_1)$ has  positive coefficients 
in terms of $G_{0}^{(k)}(x),G_1^{(k)}(x),\ldots, G_{d-1}^{(k)}(x)$. This implies that $g(x)$ has positive 
coefficients in terms of $F_{0}^{(k)}(x),F_1^{(k)}(x),\ldots, F_{d-1}^{(k)}(x)$. 
By 
Theorem~\ref{thm:posi_coef}, it is shown that $f(x)$ has positive 
coefficients in terms of $F_{0}^{(k)}(x),F_1^{(k)}(x),\ldots, F_{2d-1}^{(k)}(x)$. 
Thus $f(x)$ satisfies the condition in Theorem~\ref{thm:lp_bound}.

By Remark~\ref{rem:attain}, $G$ attains the linear programming bound 
obtained from $f(x)$. 
Assume there exists a graph $G'$ such that 
its second-largest eigenvalue is smaller than $\tau_1$, and it has the same number of vertices as $G$.   
Then $G'$ also attains the linear programming bound obtained from $f(x)$. 
By Remark~\ref{rem:attain}, $G'$ has only $d$ distinct eigenvalues,
 and the 
girth of $G'$ is at least $2d$.  Therefore $G'$ is of diameter at least $d$, it
contradicts that the number of distinct eigenvalues of $G'$ is greater than $d$.  Thus
$G$ is an extremal expander graph. 
\end{proof}

\begin{remark}
Levenshtein~\cite{L92} proved that a spherical $s$-distance set of strength $t$ satisfying $t\geq 2s-1$ is an optimal spherical code in the sense of maximizing the minimum distance. This result is  
the dual version of Theorem~\ref{thm:expander}. 
Cohn and Kumar \cite{CK07} extended this result to universally optimal codes.   
\end{remark}

We characterize connected regular graphs  satisfying $g \geq 2d$ as follows.  
\begin{theorem} \label{thm:girth}
Let $G$ be a connected $k$-regular graph of girth $g$, and with only $d+1$ distinct eigenvalues. 
If $g \geq 2d-1$ holds, then $G$ is a distance-regular graph of diameter $d$.    
\end{theorem}
\begin{proof}
Brouwer and Haemers \cite{BH93} proved that a graph with the spectrum of a distance-regular graph with
diameter $D$ and girth at least $2D-1$, is such a graph.  The proof in \cite{BH93} used the
fact that regularity, connectedness, girth, and diameter of a graph are determined by the spectrum. Therefore
the theorem of Brouwer and Haemers is interpreted as that 
a connected regular graph with $g\geq 2D-1$ is distance-regular. In general we have $d\geq D$.  Therefore in our condition,  
$g\geq 2d-1 \geq 2D-1$ holds, and  $G$ is distance-regular. Since $G$ is distance-regular, $d$ is equal to the diameter \cite[Section~4.1]{BCNb}.
\end{proof}
 Abiad, Van Dam, and Fiol \cite{AVFx}  proved Theorem~\ref{thm:girth} independently. 
\begin{remark}
Delsarte, Goethals, and Seidel~\cite{DGS77} proved that a spherical $s$-distance set of strength $t$ satisfying $t\geq 2s-2$ has the structure of 
a $Q$-polynomial association scheme. This result is  
the dual version of Theorem~\ref{thm:girth}. 
\end{remark}
The distance-regular graph with $g\geq 2d$ is called a {\it Moore polygon} \cite{DG81} and it has the following intersection array: 
\[
\begin{pmatrix}
\ast & 1    & 1  &   \cdots & 1 &c \\
0     & 0    & 0  &   \cdots &0 & k-c \\
k     & k-1 &k-1& \cdots & k-1 & \ast 
\end{pmatrix},
\]
where $c$ is a natural number. 
If $c=1$, then the graph is a Moore graph and 
it does not exist for $d\geq 3$ (with $k \geq 3$) \cite{BI73,D73}.  
If $c=k$, then the graph is an incidence graph of a regular generalized 
polygon of order $(s,s)$ \cite[Section 6.9]{BCNb},  
and it does not exist for $d\geq 7$ (with $k \geq 3$) \cite{FH64}. 
If $c\ne 1,k$, then the graph is called a {\it non-trivial} Moore polygon, and it does not exist for $d\geq 6$ \cite{DG81}. Strongly regular graphs of girth $4$
are non-trivial Moore polygons.

\begin{table} 
\caption{Extremal expander graphs}\label{tab:2}
\begin{tabular}{ccccc}
$v$    &$k$   & $g$& Eigenvalues & Name \\ \hline
$g$    &$2$    & $g$ & $2\cos(2k\pi/n); 1 \leq k \leq n-1$&$g$-cycle $C_g$ \\
$k+1$ & $k$  & $3$&$0$ & complete $K_{k+1}$\\
$2k$   & $k$   &  $4$&$0,-k$ & comp.\ bipartite  $K_{k,k}$ \\
$2\sum_{i=0}^2 q^i$& $q+1$     &  $6$& $\pm \sqrt{q},-(q+1)$& inc.\ graph of $PG(2,q)$ \cite{S66,BCNb} \\
$2\sum_{i=0}^3 q^i$&  $q+1$     &  $8$ &$\pm \sqrt{2q},0,-(q+1)$  &inc.\ graph of $GQ(q,q)$ \cite{B66,BCNb} \\
$2\sum_{i=0}^5 q^i $ &  $q+1$     &   $12$&
$\pm\sqrt{3q}, \pm \sqrt{q}, 0, -(q+1)$  &inc.\ graph of $GH(q,q)$ \cite{B66,BCNb}\\
$10$ &3&  5 & $1^5$, $(-2)^4$&Petersen \cite{HS60}\\
$50$ &7 & 5 &$2^{28}$, $(-3)^{21}$ &Hoffman--Singleton \cite{HS60} \\
$35$ &4 & 6 & $2^{14}$, $(-1)^{14}$, $(-3)^{6}$ & Odd graph \cite{M82} \\
$16$ &5 &  4 &$1^{10}$, $(-3)^5$ &Clebsch \cite{S68,G95} \\
 $56$ &10 &4 &$2^{35}$, $(-4)^{20}$ &Gewirtz \cite{BH93,G95} \\
$77$ &16 &4& $2^{55}$, $(-6)^{21}$&$M_{22}$ \cite{HS68,G95} \\
$100$ &22  &4 & $2^{77}$, $(-8)^{22}$ &Higman--Sims \cite{HS68,G95}
\end{tabular}\\
$PG(2,q)$: projective plane, $GQ(q,q)$: generalized quadrangle, \\
$GH(q,q)$: generalized hexagon, $q$: prime power 
\end{table}

Table \ref{tab:2} shows  known examples of extremal expander graphs satisfying $g \geq 2d$. 
The following graphs are unique:
Petersen graph \cite{HS60}, 
Hoffman--Singleton graph \cite{HS60}, 
Odd graph \cite{M82}, 
Clebsch graph  \cite[Theorem~10.6.4]{GRb},  
Gewirtz graph \cite{G69}, 
$M_{22}$ graph \cite{B83},
Higman--Sims graph \cite{G69}, 
$PG(2,q)$ for $q \leq 8$ \cite{M07,H53,HSW56}, 
$GQ(q,q)$ for $q \leq 4$ \cite{P75,P77}, and $GH(2,2)$ \cite{CT85}.  
For $PG(2,9)$, there are four non-isomorphic graphs \cite{HSK59,VM07}.  The uniqueness of other examples in Table \ref{tab:2} is open. 

\bigskip

\noindent
\textbf{Acknowledgments.} 
The author would like to thank Noga Alon, Eiichi Bannai, Tatsuro Ito, Jack Koolen, Hajime Tanaka,  Sho Suda, and Masato Mimura for useful information and comments.  
%This work was supported in part by JSPS Grants-in-Aid for Scientific Research No.\ 25800011. 
The author is supported by JSPS KAKENHI Grant Numbers 25800011, 26400003.

\end{document}